\documentclass[a4paper,11pt]{article}

\usepackage[T1]{fontenc}
\usepackage[utf8]{inputenc} 
\usepackage[english]{babel}
\usepackage{microtype}

\usepackage{amsmath,amsthm,amsfonts,amssymb,mathrsfs}
\usepackage{mathtools}

\usepackage{graphicx}
\usepackage[a4paper,margin=3cm]{geometry}

\usepackage{authblk}

\usepackage[hidelinks]{hyperref}

\usepackage{xparse}
\usepackage{color}

\newcommand{\dd}{\mathrm{d}}

\theoremstyle{plain}
\newtheorem{theorem}{Theorem}[section]
\newtheorem{lemma}[theorem]{Lemma}

\theoremstyle{definition}

\theoremstyle{remark}
\newtheorem{remark}[theorem]{Remark}

\numberwithin{equation}{section}

\hypersetup{
  pdftitle={Redshift Suppression of Nonlinear Scalar Fields in Accelerating FLRW Spacetimes},
  pdfauthor={Mirda Prisma Wijayanto and coauthors},
  pdfkeywords={FLRW, nonlinear scalar field, redshift suppression, weighted energy},
}

\begin{document}
\allowdisplaybreaks
\pagestyle{plain}

\title{Redshift Suppression of Nonlinear Scalar Fields in Accelerating FLRW Spacetimes}

\author[1]{Mirda Prisma Wijayanto\thanks{Corresponding author: \href{mailto:mirda.wijayanto@unsoed.ac.id}{mirda.wijayanto@unsoed.ac.id}}}
\author[2]{Emir Syahreza Fadhilla}
\author[2]{Ryan Sugihakim}
\author[3]{Iwan Setiawan}
\author[2]{Fiki Taufik Akbar}
\author[2]{Bobby Eka Gunara}

\affil[1]{Department of Physics, Faculty of Mathematics and Natural Sciences, Universitas Jenderal Soedirman, Jl.\ Dr.\ Soeparno 61, Purwokerto 53122, Indonesia}
\affil[2]{Theoretical High Energy Physics Research Division, Faculty of Mathematics and Natural Sciences, Institut Teknologi Bandung, Jl.\ Ganesha No.\ 10, Bandung 40132, Indonesia}
\affil[3]{Department of Physics Education, University of Bengkulu, Kandang Limun, Bengkulu 38371, Indonesia}

\date{} 
\maketitle

\begin{abstract}
We study small--data solutions of a nonlinear scalar field equation on spatially flat $d$--dimensional FLRW spacetimes ($d\ge4$). In conformal time $\tau$ the field satisfies a damped semilinear wave/Klein--Gordon equation with time--dependent coefficients determined by the scale factor $a(\tau)$ and the conformal Hubble rate $H(\tau)=\dot a/a$. We focus on accelerated conformal expansion of the form $H(\tau)=H_0(1+\tau)^{-\alpha}$ with $H_0>0$ and $0\le\alpha<1$, for which $a(\tau)$ grows stretched--exponentially, and we assume a power potential $V(\varphi)=-\frac{\varepsilon}{m+1}|\varphi|^{m+1}$. For global solutions arising from sufficiently small, spatially localized initial data, we introduce the conformal rescaling $\phi=a^{(d-2)/2}\varphi$, which removes the first--order Hubble damping and exposes the interaction as a \emph{time--dependent coupling}. In the rescaled equation the nonlinearity is weighted by $g(\tau)=a(\tau)^{\sigma}$ with
\[
\sigma=\frac{d+2-(d-2)m}{2},
\]
so the conformal power $m_{\mathrm{conf}}=\frac{d+2}{d-2}$ is the sharp threshold for redshift suppression: $g$ decays for $m>m_{\mathrm{conf}}$, is constant for $m=m_{\mathrm{conf}}$ (classical conformal invariance), and grows for $1<m<m_{\mathrm{conf}}$. For accelerated conformal expansion $H(\tau)=H_0(1+\tau)^{-\alpha}$ with $0\le\alpha<1$ and for superconformal interactions $m>m_{\mathrm{conf}}$, we prove that $g\in L^1([0,\infty))$ and deduce small--data global existence together with scattering/asymptotic linearization for $\phi$. As a complementary result in the diffusion--dominated regime $1<m<1+\frac{2}{d-1}$, we adapt a weighted energy method for variable damping to deduce explicit $L^2$ and $L^1$ decay rates of the form
\[
\|\varphi(\tau,\cdot)\|_{L^2}\lesssim a(\tau)^{-2/(m-1)}(1+\tau)^{-(1+\alpha)(\frac{1}{m-1}-\frac{d-1}{4})}
\]
\[
\|\varphi(\tau,\cdot)\|_{L^1}\lesssim a(\tau)^{-2/(m-1)}(1+\tau)^{-(1+\alpha)(\frac{1}{m-1}-\frac{d-1}{2})}.
\]
These bounds provide a quantitative PDE formulation of redshift--induced suppression of nonlinear scalar self--interactions at late conformal times.
\end{abstract}


\section{Introduction}
The nonlinear scalar field in the accelerating FLRW background is a fundamental component of modern cosmology because it simultaneously mediates the dynamics of early inflation, late cosmic acceleration, and the spectrum of primordial fluctuations observed in CMB anisotropies and large-scale structures \cite{Guth,Mukhanov,Frusciante}. In the effective quantum field framework in curved spacetime, the nonlinear self-interaction of the scalar field is not a small correction but rather carries information about vacuum stability, structure formation, and the survival of bound states under cosmological redshift \cite{Weinberg}.

Phenomenologically, the nonlinear potential determines whether the inflaton or dark-energy field decays toward an effective free field state or retains intrinsic dynamics that can leave measurable traces in the power spectrum and non-Gaussianity \cite{Bartolo,Chen}. From a mathematical perspective, accelerated expansion produces Hubble damping and amplitude dilution that compete directly with nonlinear amplification, so the question of whether and how nonlinearity persists over long timescales becomes a sharp PDE realization of the cosmic no–hair principle \cite{Wald,Ringstrom}. Therefore, studying nonlinear scalar fields in rapidly expanding FLRW is not only a matter of modeling inflation but also provides a rigorous framework for understanding how cosmic geometry controls, or even suppresses, microscopic nonlinear dynamics on cosmological scales.

Over the past two decades, the dynamics of scalar fields in an expanding FLRW background have been systematically analyzed through the theory of damped wave equations and hyperbolic PDEs with time-dependent coefficients \cite{Todorova,Nishihara}. In particular, the work of Todorova-Yordanov \cite{Todorova} identified the critical exponent and decay rate for damped wave equations with power nonlinearities, which later became the mathematical foundation for understanding diffusive phenomena in cosmological wave equations. Further developments by Nishihara and Zhai \cite{Nishihara} show that for slowly decaying time-dependent damping, such as that naturally arising from the FLRW expansion, the solution enters an asymptotic regime resembling the heat equation (diffusion phenomenon). Parallel to this, in relativistic cosmology, linear scalar and metric perturbations on accelerated de Sitter and FLRW backgrounds have been shown to be dynamically suppressed and driven toward a homogeneous de Sitter-like state, a manifestation of the cosmic no-hair theorem for linearized fluctuations \cite{Kitada,Bruni,Kaloper}. These results demonstrate that inflationary expansion erases anisotropies and inhomogeneities at the level of test fields and perturbations.

Throughout this paper, however, the expansion rate is measured in conformal time. The resulting dynamics should therefore be interpreted as an effective damping geometry induced by accelerated FLRW expansion, rather than as physical de Sitter inflation in cosmic time. While semilinear wave or Klein-Gordon dynamics on (asymptotically) de Sitter backgrounds have been analyzed via Strichartz and microlocal methods \cite{Hintz,Baskin,Nakamura,Yagdjian,Palmieri}, related asymptotic linearization results were obtained in these works, but without isolating the decay of the nonlinear coupling itself. 
By contrast, after the conformal rescaling $\phi=a^{(d-2)/2}\varphi$ the interaction becomes a \emph{time--dependent coupling} $g(\tau)=a(\tau)^{\sigma}$ with $\sigma=\frac{d+2-(d-2)m}{2}$. Thus the conformal power $m_{\mathrm{conf}}=\frac{d+2}{d-2}$ is the sharp threshold for redshift suppression: $g$ decays for $m>m_{\mathrm{conf}}$, is constant for $m=m_{\mathrm{conf}}$, and grows for $1<m<m_{\mathrm{conf}}$. In the superconformal regime $m>m_{\mathrm{conf}}$ we prove that $g\in L^1([0,\infty))$ for $0\le\alpha<1$ and deduce small--data global existence and scattering/asymptotic linearization for $\phi$; as a complementary result, we derive weighted $L^1$ and $L^2$ decay estimates for $\varphi$ in a diffusion--dominated range of exponents. These approaches provide a rigorous formulation of redshift-induced decay estimates for nonlinear scalar field interactions, demonstrating that accelerated expansion dynamically weakens nonlinear self-interactions and suggests the suppression of localized nonlinear structures in the small-data regime in inflationary FLRW spacetimes. This provides a nonlinear PDE analog of the cosmic no-hair mechanism at the level of scalar self-interactions.

The remainder of the paper is organized as follows. In Section \ref{Sec2}, we introduce the nonlinear scalar field model on higher-dimensional spatially flat FLRW spacetimes and derive the corresponding evolution equations. In Section \ref{Sec3}, we establish the main analytical results of the paper. We first introduce the conformal rescaling and identify the effective time-dependent nonlinear coupling generated by cosmological redshift. We then prove small-data global existence and scattering/asymptotic linearization in the superconformal regime. As a complementary result, we derive weighted $L^1$- and $L^2$-decay estimates in the diffusion-dominated range by means of a weighted energy method adapted to variable damping. We also include a brief discussion of the intermediate exponent range $1+\frac{2}{d-1}<m<m_{\mathrm{conf}}$,
which is not covered by either the superconformal scattering argument or the diffusion-dominated weighted-energy method. The final Section \ref{Sec4} is devoted to the implications of these results for inflationary stability and nonlinear field dynamics in expanding universes.

\section{Nonlinear Scalar Fields on FLRW Spacetimes}\label{Sec2}
We consider a $d$–dimensional spatially flat Friedmann-Lemaître-Robertson-Walker (FLRW) spacetime 
$\mathcal{M}^d$, $d\ge 4$, endowed with global coordinates $x^\mu=(t,x^i)$, $\mu=0,1,\dots,d-1$, 
and a Lorentzian metric of signature $(-,+,\dots,+)$
\begin{align}\label{metric1}
ds^2=-dt^2+a^2(t)\sum_{i=1}^{d-1}dx_i^2,
\end{align}
where $x^i$ are Cartesian coordinates on $\mathbb{R}^{d-1}$, and $a(t)$ denotes the cosmological scale factor.

Introducing the conformal time $\tau$,
\begin{align}\label{conformal}
\frac{d\tau}{dt}=\frac{1}{a(t)},
\end{align}
the metric takes the conformally flat form
\begin{align}\label{metric2}
ds^2=a^2(\tau)\big(-d\tau^2+dx^2\big),
\end{align}
so that $\mathcal{M}^d$ is conformal to Minkowski space $\mathbb{R}\times\mathbb{R}^{d-1}$.
We assume $a\in C^2((0,\infty))$ and define the Hubble parameter
\begin{align}
H(\tau)=\frac{\dot a(\tau)}{a(\tau)},\qquad \dot a=\frac{da}{d\tau}.
\end{align}
The Ricci tensor and scalar curvature associated with \eqref{metric2} are given by
\begin{align}
R_{\mu\nu} =& \dot{H}\left(\eta_{\mu\nu}-(d-2)\delta^0_\mu\delta^0_\nu\right) + H^2(d-2)(\eta_{\mu\nu}+\delta^0_\mu\delta^0_\nu), \label{Ricci tensor}\\
R =& (d-1)a^{-2}\left(2\dot{H}+(d-2)H^2\right) \label{Curvature scalar}.
\end{align}

We study a real scalar field $\varphi$ propagating on this expanding background with a non-minimal coupling to curvature. Its dynamics are governed by the action
\begin{align}\label{action}
S=\int d\tau\,dx\,\sqrt{-g}\left(\frac12\,\partial_\mu\varphi\,\partial^\mu\varphi
+\frac{\xi}{2}R\varphi^2-V(\varphi)\right),
\end{align}
where $\xi>0$ is the non–minimal coupling constant and $V\in C^\infty(\mathbb{R})$ satisfies
$V(0)=V'(0)=0$.
The corresponding Euler-Lagrange equation reads
\begin{align}\label{eom}
\nabla_\mu\nabla^\mu\varphi-\xi R\varphi+V'(\varphi)=0.
\end{align}
Substituting \eqref{Ricci tensor}-\eqref{Curvature scalar} into \eqref{eom}, we obtain the nonlinear damped wave equation
\begin{align}\label{nlwave}
(\partial_\tau^2-\Delta)\varphi
=F(\varphi,\partial_\tau\varphi),
\end{align}
where
\begin{align}
F(\varphi,\partial_\tau\varphi) \equiv -(d-2)H\partial_\tau\varphi - \xi(d-1)(2\dot{H}+(d-2)H^2)\varphi + a^2(\tau)\partial_\varphi V(\varphi).
\end{align}

Let $f\in H^{l+1}(\mathbb{R}^{d-1})$ and $g\in H^l(\mathbb{R}^{d-1})$ be compactly supported and let $l>\frac{1}{2}(d-1)$. The Cauchy problem reads \begin{equation}
\label{Eq.1}
\begin{cases}
\partial_\tau^2\varphi - \Delta\varphi + (d-2)H\partial_\tau\varphi + \xi(d-1)(2\dot{H}+(d-2)H^2)\varphi - a^2\partial_\varphi V(\varphi) = 0,\\ \varphi(\tau_0,x)=f(x),\qquad \partial_\tau\varphi(\tau_0,x)=g(x). \end{cases}
\end{equation}
It was proved in \cite{fiki} that for sufficiently small initial data in $H^{l+1}(\mathbb{R}^{d-1})\times H^l(\mathbb{R}^{d-1})$, the Cauchy problem \eqref{Eq.1} admits a unique global classical solution
\begin{align} \varphi\in C\big([T_0,\infty),H^{l+1}\big)\cap C^1\big([T_0,\infty),H^l\big).
\end{align}

The purpose of the present work is to rigorously quantify how cosmological expansion suppresses nonlinear scalar self-interactions through cosmological redshift. To this end, we assume that the expansion rate is of the form
\begin{align}\label{Halpha}
H(\tau)=H_0(1+\tau)^{-\alpha},\qquad H_0>0,\quad 0\leq \alpha<1,
\end{align}
and that the scalar field is governed by a power-law potential
\begin{align}\label{V}
V(\varphi)=-\frac{\varepsilon}{m+1}|\varphi|^{m+1},\qquad \varepsilon>0,\quad m>1.
\end{align}
\begin{remark}
When $m$ is not an odd integer, the map $\varphi\mapsto |\varphi|^{m+1}$ is not $C^\infty$ at $\varphi=0$.
In the PDE arguments below we only require the nonlinearity $\varphi\mapsto |\varphi|^{m-1}\varphi$ to be locally Lipschitz in the Sobolev spaces under consideration.
In the scattering result (Theorem~\ref{Thm:scattering}) we assume $m$ to be an odd integer to remain within the smooth conformal field setting.
\end{remark}
The restriction $0\leq \alpha<1$ selects the physically relevant class of accelerating FLRW cosmologies. Indeed, since $H(\tau)=\dot a(\tau)/a(\tau)$, the ansatz above yields
\begin{align}\label{a(tau)}
a(\tau)=a_0\exp\left(\frac{H_0}{1-\alpha}\big((1+\tau)^{1-\alpha}-1\big)\right),
\qquad \alpha\neq1,
\end{align}
where $a_0>0$ denotes the initial value of the scale factor. 
The resulting scale factor exhibits stretched–exponential growth in conformal time for $0\le \alpha<1$, reflecting a rapidly expanding conformal geometry. It is important to note that $H(\tau)=\dot a(\tau)/a(\tau)$ is the \emph{conformal} Hubble rate, whereas the physical Hubble parameter in cosmic time $t$ is $H_{\mathrm{phys}}(t)=\frac{1}{a(t)}\frac{\dd a(t)}{\dd t}$; these are related by $H_{\mathrm{phys}}(t)=H(\tau)/a(\tau)$. Consequently, the case $\alpha=0$ describes a spacetime with strong conformal redshift and effective damping. Such fast conformal expansion is nevertheless sufficient to generate a pronounced cosmological redshift that drives the spacetime toward isotropy and homogeneity at late times, in the sense of the cosmic no-hair mechanism \cite{Kitada,Kaloper}.
For $\alpha<0$, the conformal expansion becomes super–accelerated and typically leads to pathologies, while for $\alpha\ge1$, the growth of $a(\tau)$ is only power-like or slower, so that redshift effects are too weak to control nonlinear interactions. We therefore focus on $0\le \alpha<1$ as the natural regime in which accelerated conformal expansion yields a systematic decay of nonlinear scalar self-interactions.

Under the assumptions \eqref{Halpha}, \eqref{V}, and \eqref{a(tau)}, \eqref{Eq.1} becomes
\begin{equation}\label{Eq.2}
\begin{cases}
\partial_\tau^2\varphi-\Delta\varphi+(d-2)H\partial_\tau\varphi
+\xi(d-1)(2\dot H+(d-2)H^2)\varphi
+a^2\varepsilon|\varphi|^{m-1}\varphi=0,\\
\varphi(\tau_0,x)=f(x),\qquad \partial_\tau\varphi(\tau_0,x)=g(x).
\end{cases}
\end{equation}
This will be the starting point of our analysis in the next section.

\section{Nonlinear Decay in Accelerating FLRW Spacetimes}\label{Sec3}
We investigate the long--time behavior of solutions to the scalar field equation \eqref{Eq.2} on an accelerating FLRW background. Our main point is that accelerated expansion affects nonlinear dynamics through two complementary mechanisms: (i)~\emph{redshift renormalization} of the interaction strength, made explicit by a conformal rescaling that turns the nonlinearity into a time--dependent coupling, and (ii)~\emph{diffusion phenomena} generated by the (non--integrable) Hubble damping. In the superconformal regime (powers above the conformal threshold) the effective coupling becomes integrable in conformal time, which yields small--data global existence and scattering/asymptotic linearization for the rescaled field. In addition, in the diffusion--dominated regime we obtain quantitative $L^{1}$ and $L^{2}$ decay rates via a weighted energy method adapted to variable damping.

\subsection{Conformal rescaling and the effective coupling}\label{Sec3Rescaling}
A key structural feature of \eqref{Eq.2} is that the Hubble damping can be removed by a conformal rescaling.
Define
\begin{equation}\label{Eq:phiDef}
\phi(\tau,x):=a(\tau)^{\frac{d-2}{2}}\varphi(\tau,x),
\qquad 
\xi_{\mathrm{c}}:=\frac{d-2}{4(d-1)}.
\end{equation}
A straightforward computation (see, e.g., the discussion around conformal coupling in \cite{Baskin,Hintz}) yields the rescaled equation
\begin{equation}\label{Eq:rescaled}
\partial_\tau^2\phi-\Delta\phi
+(d-1)(\xi-\xi_{\mathrm{c}})\big(2\dot H+(d-2)H^2\big)\phi
+\varepsilon\,a(\tau)^{\sigma}|\phi|^{m-1}\phi=0,
\end{equation}
where the \emph{redshift exponent} is
\begin{equation}\label{Eq:sigma}
\sigma:=\frac{d+2-(d-2)m}{2}.
\end{equation}
Thus the nonlinear interaction is governed by the \emph{time--dependent coupling} $g(\tau)=a(\tau)^{\sigma}$.
The \emph{conformal power}
\begin{equation}\label{Eq:mconf}
m_{\mathrm{conf}}:=\frac{d+2}{d-2}
\end{equation}
is characterized by $\sigma=0$ (classical conformal invariance of the massless, conformally coupled theory).

\begin{lemma}[Integrability of the redshift coupling]\label{Lemma:coupling_integrable}
Assume \eqref{Halpha} (equivalently \eqref{a(tau)}) with $0\le\alpha<1$.
If $\sigma<0$ (equivalently $m>m_{\mathrm{conf}}$), then $g\in L^1([0,\infty))$ and
\begin{equation}\label{Eq:coupling_tail}
\int_\tau^\infty a(s)^{\sigma}\,ds \;\lesssim\; (1+\tau)^{\alpha}\,a(\tau)^{\sigma},
\qquad \tau\ge0.
\end{equation}
\end{lemma}
\begin{proof}
Using \eqref{a(tau)}, we can write $a(\tau)^{\sigma}=a_0^{\sigma}\exp\!\big(\kappa\sigma\big((1+\tau)^{1-\alpha}-1\big)\big)$ with $\kappa:=\frac{H_0}{1-\alpha}>0$.
Let $u=(1+s)^{1-\alpha}$ so that $ds=\frac{1}{1-\alpha}u^{\frac{\alpha}{1-\alpha}}\,du$.
Then for $\tau\ge0$,
\[
\int_\tau^\infty a(s)^{\sigma}\,ds
=\frac{a_0^{\sigma}e^{-\kappa\sigma}}{1-\alpha}
\int_{(1+\tau)^{1-\alpha}}^\infty u^{\frac{\alpha}{1-\alpha}}e^{\kappa\sigma u}\,du.
\]
Since $\sigma<0$, the integral converges and standard estimates for incomplete Gamma functions imply
\[
\int_{u_0}^\infty u^{p}e^{-c u}\,du \;\lesssim\; u_0^{p}e^{-c u_0}\quad (u_0\ge1),
\]
with $c=-\kappa\sigma>0$ and $p=\frac{\alpha}{1-\alpha}$.
Translating back to $\tau$ yields \eqref{Eq:coupling_tail}, and in particular $a(\cdot)^{\sigma}\in L^1([0,\infty))$.
\end{proof}

\subsection{Superconformal regime: global existence and scattering}\label{Sec3Scattering}
We next exploit Lemma~\ref{Lemma:coupling_integrable} in the \emph{superconformal} range $m>m_{\mathrm{conf}}$, where the interaction becomes short--range in conformal time.
For clarity we impose the conformal coupling $\xi=\xi_{\mathrm{c}}$, so that the linear part of \eqref{Eq:rescaled} reduces to the flat wave operator.
(For $\xi\neq\xi_{\mathrm{c}}$, one obtains scattering relative to the linear equation with the time--dependent potential in \eqref{Eq:rescaled}; see Remark~\ref{Rem:nonconformal_scattering} below.)

\begin{theorem}[Short--range nonlinearity and scattering]\label{Thm:scattering}
Let $d\ge4$, $0\le\alpha<1$, and assume $\xi=\xi_{\mathrm{c}}$.
Assume that $m>m_{\mathrm{conf}}=\frac{d+2}{d-2}$ and that $m$ is an odd integer so that $u\mapsto |u|^{m-1}u$ is smooth.
Let $l>\frac12(d-1)$ and consider initial data $(f,g)\in H^{l+1}(\mathbb{R}^{d-1})\times H^{l}(\mathbb{R}^{d-1})$ at $\tau_0=0$.
Let $(\phi_0,\phi_1)$ be the corresponding rescaled data
\begin{equation}\label{Eq:rescaledData}
\phi_0=a(0)^{\frac{d-2}{2}}f,\qquad
\phi_1=a(0)^{\frac{d-2}{2}}g+\frac{d-2}{2}H(0)a(0)^{\frac{d-2}{2}}f.
\end{equation}
There exists $\delta>0$ such that if $\|(\phi_0,\phi_1)\|_{H^{l+1}\times H^{l}}\le\delta$, then the Cauchy problem \eqref{Eq:rescaled} admits a unique global solution $\phi$ on $[0,\infty)$ with
\begin{equation}\label{Eq:uniformSmall}
\sup_{\tau\ge0}\|(\phi(\tau),\partial_\tau\phi(\tau))\|_{H^{l+1}\times H^{l}}\;\lesssim\;\delta.
\end{equation}
Moreover, there exist unique scattering states $(\phi_+,\phi_+')\in H^{l+1}\times H^{l}$ such that
\begin{equation}\label{Eq:scattering}
\|(\phi(\tau),\partial_\tau\phi(\tau)) - S(\tau)(\phi_+,\phi_+')\|_{H^1\times L^2}
\;\lesssim\; \delta^{m}\int_\tau^\infty a(s)^{\sigma}\,ds
\;\xrightarrow[\tau\to\infty]{}\;0,
\end{equation}
where $S(\tau)$ denotes the free wave group on $\mathbb{R}^{d-1}$.
Consequently, the original field satisfies the asymptotically linear decay
\begin{equation}\label{Eq:L2LinearDecay}
\|\varphi(\tau,\cdot)\|_{L^2(\mathbb{R}^{d-1})}
= a(\tau)^{-\frac{d-2}{2}}\|\phi(\tau,\cdot)\|_{L^2(\mathbb{R}^{d-1})}
\;\lesssim\; \delta\,a(\tau)^{-\frac{d-2}{2}},\qquad \tau\ge0.
\end{equation}
\end{theorem}
\begin{proof}
Let $n=d-1$.
Since $l>\frac{n}{2}$, $H^l(\mathbb{R}^n)$ is an algebra and $H^{l+1}(\mathbb{R}^n)\hookrightarrow L^\infty(\mathbb{R}^n)$.
With $\xi=\xi_{\mathrm{c}}$, equation \eqref{Eq:rescaled} reduces to
\[
\partial_\tau^2\phi-\Delta\phi=-\varepsilon\,a(\tau)^{\sigma}|\phi|^{m-1}\phi.
\]
Let $S(\tau)$ be the free wave evolution on $H^{l+1}\times H^{l}$, which is unitary on $H^{l+1}\times H^{l}$ and on $H^1\times L^2$.
The mild formulation reads
\begin{equation}\label{Eq:Duhamel}
(\phi(\tau),\phi_\tau(\tau))
= S(\tau)(\phi_0,\phi_1)
-\varepsilon\int_0^\tau S(\tau-s)\bigl(0,\,a(s)^{\sigma}|\phi(s)|^{m-1}\phi(s)\bigr)\,ds.
\end{equation}

\smallskip
\noindent\emph{Global existence.}
Let $X:=\sup_{\tau\ge0}\|(\phi(\tau),\phi_\tau(\tau))\|_{H^{l+1}\times H^{l}}$.
Using the algebra property of $H^l$ and Sobolev embedding,
\[
\bigl\||\phi|^{m-1}\phi\bigr\|_{H^{l}}
\;\lesssim\;\|\phi\|_{L^\infty}^{m-1}\|\phi\|_{H^{l}}
\;\lesssim\;\|\phi\|_{H^{l+1}}^{m}.
\]
Taking the $H^{l+1}\times H^{l}$ norm in \eqref{Eq:Duhamel} yields
\[
X \;\lesssim\; \|(\phi_0,\phi_1)\|_{H^{l+1}\times H^{l}} + X^{m}\int_0^\infty a(s)^{\sigma}\,ds.
\]
By Lemma~\ref{Lemma:coupling_integrable}, $I:=\int_0^\infty a(s)^{\sigma}\,ds<\infty$.
Choosing $\delta>0$ so that $C(\delta^{m-1}I)\ll1$ and applying a standard bootstrap/continuation argument gives \eqref{Eq:uniformSmall}.

\smallskip
\noindent\emph{Scattering.}
Apply $S(-\tau)$ to \eqref{Eq:Duhamel}:
\[
S(-\tau)(\phi(\tau),\phi_\tau(\tau))
=(\phi_0,\phi_1)-\varepsilon\int_0^\tau S(-s)\bigl(0,\,a(s)^{\sigma}|\phi(s)|^{m-1}\phi(s)\bigr)\,ds.
\]
The integrand is bounded in $H^1\times L^2$ by
\[
\bigl\|a(s)^{\sigma}|\phi(s)|^{m-1}\phi(s)\bigr\|_{L^2}
\;\lesssim\; a(s)^{\sigma}\|\phi(s)\|_{L^\infty}^{m-1}\|\phi(s)\|_{L^2}
\;\lesssim\; a(s)^{\sigma} X^{m}.
\]
Since $a(\cdot)^{\sigma}\in L^1([0,\infty))$, the integral converges in $H^1\times L^2$ as $\tau\to\infty$, so the limit
\[
(\phi_+,\phi_+'):=\lim_{\tau\to\infty} S(-\tau)(\phi(\tau),\phi_\tau(\tau))
\]
exists in $H^1\times L^2$ and in fact in $H^{l+1}\times H^{l}$ by the same argument.
Finally, subtracting the limit from the representation above gives \eqref{Eq:scattering} with the rate controlled by the tail integral $\int_\tau^\infty a(s)^{\sigma}\,ds$.
The $L^2$--decay \eqref{Eq:L2LinearDecay} follows from $\varphi=a^{-(d-2)/2}\phi$ and \eqref{Eq:uniformSmall}.
\end{proof}

\begin{remark}[Marginal conformal interaction]\label{Rem:marginal}
If $m=m_{\mathrm{conf}}$ and $\xi=\xi_{\mathrm{c}}$, then $\sigma=0$ and \eqref{Eq:rescaled} reduces exactly to the flat semilinear wave equation
\[
\partial_\tau^2\phi-\Delta\phi=-\varepsilon|\phi|^{m-1}\phi
\quad\text{on }\mathbb{R}^{d-1}.
\]
Thus any well--posedness and scattering statement for the Minkowski equation transfers \emph{verbatim} to the FLRW model via $\varphi=a^{-(d-2)/2}\phi$.
In particular, in the physically relevant case $d=4$ and $m=3$ (classical $\lambda\varphi^4$ theory), the rescaled equation is the defocusing cubic wave equation, for which global well--posedness in the energy class and small--data scattering follow from standard energy and Strichartz estimates.
\end{remark}

\begin{remark}[Nonconformal coupling]\label{Rem:nonconformal_scattering}
If $\xi\neq\xi_{\mathrm{c}}$, the rescaled equation \eqref{Eq:rescaled} contains the time--dependent potential $(d-1)(\xi-\xi_{\mathrm{c}})(2\dot H+(d-2)H^2)$.
The same argument yields scattering relative to the corresponding linear propagator.
In addition, if $2\dot H+(d-2)H^2\in L^1([0,\infty))$ (for instance if $\alpha>\tfrac12$), then the linear equation itself scatters to a free wave, and \eqref{Eq:scattering} implies free scattering for $\phi$.
\end{remark}

\subsection{Diffusion regime: weighted energy decay}\label{Sec3Diffusion}
We now return to the original variable $\varphi$ and complement the scattering picture above with decay estimates in the diffusion--dominated range of exponents for which weighted energy methods are effective.
\begin{theorem}\label{Thm:diffusion}
Let $d\ge4$ and $1<m<1+\frac{2}{d-1}$. 
Assume that the conformal Hubble rate satisfies
\begin{align}
H(\tau)=H_0(1+\tau)^{-\alpha},\qquad H_0>0,\qquad 0\leq \alpha<1.
\end{align}
Let $\varphi$ be a global solution of \eqref{Eq.2} with spatially localized (e.g.\ compactly supported) initial data $(f,g)\in H^{l+1}(\mathbb{R}^{d-1})\times H^l(\mathbb{R}^{d-1})$, where $l>\frac12(d-1)$, and assume the data are sufficiently small so that global existence holds (cf.\ \cite{fiki}). Then there exists a constant $C>0$, depending on the size of the initial data, such that for all $\tau\ge0$,
\begin{align}
\|\varphi(\tau,\cdot)\|_{L^1(\mathbb{R}^{d-1})}\leq{}& C\,a(\tau)^{-\frac{2}{m-1}}(1+\tau)^{-(1+\alpha)\left(\frac{1}{m-1}-\frac{d-1}{2}\right)}, \label{L1decay}\\
\|\varphi(\tau,\cdot)\|_{L^2(\mathbb{R}^{d-1})}\leq{}& C\,a(\tau)^{-\frac{2}{m-1}}(1+\tau)^{-(1+\alpha)\left(\frac{1}{m-1}-\frac{d-1}{4}\right)}. \label{L2decay}
\end{align}
\end{theorem}
\begin{proof}
	We prove the theorem using a weighted energy method originally developed by \cite{Todorova}, which has proven effective in capturing diffusion-type decay for damped wave equations. This approach has been further refined in the context of diffusion phenomena and variable damping coefficients \cite{Nishihara,Radu,LuLi,Sobajima}. Here, we introduce the weight function
	\begin{align}\label{weight function}
	\psi(\tau,x) = \frac{b|x|^2}{(T_0+\tau)^{1+\alpha}},
	\end{align}
	where $0<b\ll 1$ is a fixed small constant and $T_0>1$. A direct computation yields
	\begin{align}
	\psi_\tau = -\frac{b(1+\alpha)|x|^2}{(T_0+\tau)^{2+\alpha}},\qquad
	\nabla\psi = \frac{2bx}{(T_0+\tau)^{1+\alpha}}.
	\end{align}
	Consequently,
	\begin{align}
	\frac{|\nabla \psi|^2}{(-\psi_\tau)}=\frac{4b}{(1+\alpha)}\frac{1}{(T_0+\tau)^\alpha}\leq \frac{4b}{(1+\alpha)}\frac{1}{(1+\tau)^\alpha}=\frac{4b}{(1+\alpha)}\frac{H(\tau)}{H_0}.
	\end{align}
	
	We multiply equation \eqref{Eq.2} by $e^{2\psi}\varphi_\tau$, such that
	\begin{align} \label{Eq.3}
	e^{2\psi}\varphi_\tau\varphi_{\tau\tau} - e^{2\psi}\varphi_\tau\Delta\varphi + (d-2)He^{2\psi} |\varphi_\tau|^2 \nonumber\\
	+\xi(d-1) (2\dot{H}+(d-2)H^2)e^{2\psi}\varphi \varphi_\tau + e^{2\psi} a^2\varepsilon |\varphi|^{m-1}\varphi\varphi_\tau = 0.
	\end{align}
	A straightforward computation shows that \eqref{Eq.3} can be written in divergence form as
	\begin{align} \label{Eq.4}
	&\frac{\partial}{\partial\tau}\left[\frac{e^{2\psi}}{2}\left(|\varphi_\tau|^2 + |\nabla\varphi|^2\right)+\frac{\varepsilon a^2 e^{2\psi}}{m+1}|\varphi|^{m+1}\right]+\xi (d-1)\left(2\dot{H}+(d-2)H^2\right)e^{2\psi}\varphi\varphi_\tau \nonumber\\
	&-\nabla\cdot\left(e^{2\psi}\varphi_\tau\nabla \varphi\right) + e^{2\psi}\left[\left((d-2)H-\frac{|\nabla\psi|^2}{(-\psi_\tau)}-\psi_\tau\right)|\varphi_\tau|^2 - \frac{2a\varepsilon}{m+1}\left(a\psi_\tau+a_\tau\right)|\varphi|^{m+1}\right]\nonumber\\
	&+\frac{e^{2\psi}}{(-\psi_\tau)}\left[\psi_\tau\nabla\varphi - \varphi_\tau\nabla\psi\right]^2 = 0.
	\end{align} 
	Then, multiplication \eqref{Eq.2} with $e^{2\psi}\varphi$ gives
	\begin{align} \label{Eq.5}
	e^{2\psi}\varphi\varphi_{\tau\tau} - e^{2\psi}\varphi\Delta\varphi + (d-2)e^{2\psi}\varphi H\partial_\tau\varphi \nonumber\\
	+ \xi(d-1)e^{2\psi}\varphi (2\dot{H}+(d-2)H^2)\varphi + e^{2\psi}\varphi a^2\varepsilon |\varphi|^{m-1}\varphi = 0,
	\end{align}
	which can be rewritten as
	\begin{align} \label{Eq.6}
	&\frac{\partial}{\partial\tau}\left[e^{2\psi}\left(\varphi\varphi_\tau+\frac{(d-2)H}{2}\varphi^2\right)\right] - \nabla\cdot(e^{2\psi}\varphi\nabla\varphi)\nonumber\\
	& + e^{2\psi}\left[\left(\xi(d-1)\left(2\dot{H}+(d-2)H^2\right)
	+ (d-2)H\left(-\psi_\tau + \frac{\alpha}{2(1+\tau)}\right)
	\right)|\varphi|^2\right.\nonumber \\
	&\left. + |\nabla\varphi|^2 - 2 \psi_\tau\varphi\varphi_\tau - |\varphi_\tau|^2 + 2\varphi\nabla\psi\cdot\nabla\varphi+ \varepsilon a^2 |\varphi|^{m+1}\right] = 0.
	\end{align}
	
	We observe that the identity \eqref{Eq.6} contains the unfavorable contribution 
	$-e^{2\psi}|\varphi_\tau|^2$, which cannot be controlled directly by the positive part of the energy functional. 
	In order to compensate for this bad term and recover a coercive weighted energy estimate, 
	we introduce an additional time-dependent weight and multiply \eqref{Eq.4} by $(T_0+\tau)^\alpha$,
	\begin{align} \label{Eq.7}
	&\frac{\partial}{\partial\tau}\left[\frac{e^{2\psi}(T_0+\tau)^\alpha}{2}\left(|\varphi_\tau|^2+|\nabla\varphi|^2\right)+\frac{\varepsilon a^2 e^{2\psi}(T_0+\tau)^\alpha}{m+1}|\varphi|^{m+1}\right]-\nabla\cdot\left[(T_0+\tau)^\alpha e^{2\psi}\varphi_\tau\nabla\varphi\right] \nonumber \\
	&+\xi (d-1)\left(2\dot{H}+(d-2)H^2\right)e^{2\psi}(T_0+\tau)^\alpha\varphi\varphi_\tau + e^{2\psi}\left[(d-2)H_0 - \frac{4b}{1+\alpha}\right. \nonumber \\
	&\left.-\frac{\alpha}{2(T_0+\tau)^{1-\alpha}}-(T_0+\tau)^\alpha \psi_\tau\right]|\varphi_\tau|^2 + e^{2\psi}(T_0+\tau)^\alpha\frac{(-2\varepsilon a)}{m+1}\left(a\psi_\tau+a_\tau\right)|\varphi|^{m+1}  \nonumber \\
	& - \frac{\alpha e^{2\psi}}{(T_0 + \tau)^{1-\alpha}}\left[\frac{1}{2}|\nabla\varphi|^2 + \frac{\varepsilon a^2}{m+1}|\varphi|^{m+1}\right] + \frac{(T_0 + \tau)^\alpha e^{2\psi}}{(-\psi_\tau)}\left[\psi_\tau\nabla\varphi - \varphi_\tau\nabla \psi\right]^2 =0.
	\end{align}
	Moreover, since
	\begin{align}
	-\frac{1}{\psi_\tau}\left[\psi_\tau\nabla\varphi - \varphi_\tau \nabla\psi\right]^2 = \frac{1}{-\psi_\tau} \left[(\psi_\tau)^2|\nabla\varphi|^2 - 2\psi_\tau\varphi_\tau\nabla\psi.\nabla\varphi + |\nabla\psi|^2 |\varphi_\tau|^2\right]
	\end{align}
	and by the inequality $2ab\leq a^2+b^2$, we have
	\begin{align}
	-\frac{1}{\psi_\tau}\left[\psi_\tau\nabla\varphi - \varphi_\tau \nabla\psi\right]^2 	\geq -\psi_\tau|\nabla\varphi|^2 + \frac{|\nabla\psi|^2}{(-\psi_\tau)}|\varphi_\tau|^2.
	\end{align}
	Substituting this lower bound into \eqref{Eq.7}, we arrive at
	\begin{align} \label{Eq.8}
	&\frac{\partial}{\partial\tau}\left[\frac{e^{2\psi}(T_0+\tau)^\alpha}{2}\left(|\varphi_\tau|^2+|\nabla\varphi|^2\right)+\frac{\varepsilon a^2 e^{2\psi}(T_0+\tau)^\alpha}{m+1}|\varphi|^{m+1}\right]\nonumber\\
	&-\nabla\cdot\left((T_0+\tau)^\alpha e^{2\psi}\varphi_\tau\nabla\varphi\right)\nonumber\\
	&+\xi (d-1)\left(-\frac{2H_0\alpha}{(1+\tau)}+\frac{(d-2)H_0^2}{(1+\tau)^\alpha}\right)e^{2\psi}\varphi\varphi_\tau\nonumber\\
	& + e^{2\psi}(T_0+\tau)^\alpha \left(\frac{(d-2)H_0}{(T_0+\tau)^\alpha} -\frac{\alpha}{2(T_0+\tau)}- \psi_\tau\right)|\varphi_\tau|^2 \nonumber \\
	& -e^{2\psi}(T_0+\tau)^\alpha \left(\frac{\alpha}{2(T_0+\tau)}+\psi_\tau\right)|\nabla\varphi|^2 \nonumber\\
	&-\frac{\varepsilon a e^{2\psi} (T_0+\tau)^\alpha}{m+1} \left(2(a\psi_\tau+a_\tau) + \frac{\alpha a}{(T_0+\tau)}\right) |\varphi|^{m+1} \leq 0,
	\end{align}
	in the sense of distributions. At this stage, \eqref{Eq.8} still contains lower-order contributions,
	namely the sign–indefinite term involving $\varphi\varphi_\tau$,
	as well as the nonlinear potential term $|\varphi|^{m+1}$ with non-coercive coefficients. These terms cannot be absorbed directly by the principal
	energy terms $|\varphi_\tau|^2$ and $|\nabla\varphi|^2$, and therefore require an additional mixed space-time multiplier argument.
	
	Multiplying \eqref{Eq.6} by a positive constant $\nu > 0$, we obtain
	\begin{align} \label{Eq.9}
	&\frac{\partial}{\partial\tau}\left[e^{2\psi}\left(\nu \varphi\varphi_\tau + \nu\frac{(d-2)}{2}\frac{H_0}{(1+\tau)^\alpha}\varphi^2\right)\right] - \nabla\cdot(\nu e^{2\psi}\varphi \nabla\varphi) - \nu e^{2\psi}|\varphi_\tau|^2 + \nu e^{2\psi}|\nabla\varphi|^2 \nonumber\\
	&+e^{2\psi}\nu\left[\xi(d-1)\left(\frac{-2H_0 \alpha}{(1+\tau)^{1+\alpha}}+\frac{(d-2)H_0^2}{(1+\tau)^{2\alpha}}\right)+(d-2)\left(-\psi_\tau+\frac{\alpha}{2(1+\tau)}\right)\frac{H_0}{(1+\tau)^\alpha}\right]|\varphi|^2 \nonumber\\
	&+\nu \varepsilon a^2 e^{2\psi} |\varphi|^{m+1} + e^{2\psi} \left(-2\nu \psi_\tau \varphi \varphi_\tau + 2\nu \varphi \nabla\psi\cdot\nabla\varphi\right) = 0.
	\end{align}
	Adding \eqref{Eq.9} to the differential inequality \eqref{Eq.8} yields
	\begin{align} \label{Eq.10}
	&\frac{\partial}{\partial\tau}\left[e^{2\psi}\left(\frac{(T_0+\tau)^\alpha}{2}|\varphi_\tau|^2 + \nu \varphi\varphi_\tau + \nu \frac{(d-2)}{2}\frac{H_0}{(1+\tau)^\alpha}\varphi^2+ \frac{(T_0+\tau)^\alpha}{2}|\nabla\varphi|^2 \right.\right.\nonumber\\
	&\left.\left. + \frac{\varepsilon a^2(T_0+\tau)^\alpha}{m+1}|\varphi|^{m+1}\right)\right]-\nabla\cdot\left(e^{2\psi}\left((T_0+\tau)^\alpha\varphi_\tau\nabla\varphi+\nu \varphi \nabla\varphi\right)\right)\nonumber\\
	& + e^{2\psi}\left[\left((d-2)H_0 - \frac{\alpha}{2(T_0+\tau)^{1-\alpha}}-\nu-(T_0+\tau)^\alpha\psi_\tau\right)|\varphi_\tau|^2 \right.\nonumber\\
	&+\left(\nu - \frac{\alpha}{2(T_0+\tau)^{1-\alpha}}-\psi_\tau(T_0+\tau)^\alpha\right)|\nabla\varphi|^2  +\xi (d-1)\left(-\frac{2H_0\alpha}{(1+\tau)}+\frac{(d-2)H_0^2}{(1+\tau)^\alpha}\right)\varphi\varphi_\tau\nonumber\\
	&+\nu\left(\xi(d-1)\left(\frac{-2H_0 \alpha}{(1+\tau)^{1+\alpha}}+\frac{(d-2)H_0^2}{(1+\tau)^{2\alpha}}\right)+(d-2)\left(-\psi_\tau+\frac{\alpha}{2(1+\tau)}\right)\frac{H_0}{(1+\tau)^\alpha}\right)|\varphi|^2 \nonumber\\
	& \left. +\left(\nu \varepsilon a^2 - \frac{\varepsilon a^2 \alpha}{(m+1)(T_0+\tau)^{1-\alpha}} - \frac{2\varepsilon a}{m+1}(a\psi_\tau + a_\tau)(T_0+\tau)^\alpha \right)|\varphi|^{m+1}\right]\nonumber \\
	&+e^{2\psi}\left(-2\nu\psi_\tau\varphi\varphi_\tau + 2\nu \varphi \nabla\varphi \cdot\nabla\psi\right) \leq 0.
	\end{align}
	
	We now fix the parameter
	\begin{align}
	\nu:=\frac{(d-2)H_0}{4}.
	\end{align}
	Then, since $(T_0+\tau)^{-(1-\alpha)}\to0$ as $T_0\to\infty$,
	there exists $T_0$ sufficiently large such that for all $\tau\ge0$,
	\begin{align}
	(d-2)H_0-\frac{\alpha}{2(T_0+\tau)^{1-\alpha}}-\nu
	\geq& \frac{(d-2)H_0}{2},\\
	\nu-\frac{\alpha}{2(T_0+\tau)^{1-\alpha}}
	\geq& \frac{(d-2)H_0}{8},\\
	\varepsilon a^2\left(\nu-\frac{\alpha}{(m+1)(T_0+\tau)^{1-\alpha}}\right)
	\geq& \frac{\varepsilon a^2(d-2)H_0}{8}.
	\end{align}
	Hence,
	\begin{align}
	&\frac{\partial}{\partial\tau}\left[e^{2\psi}\left(\frac{(T_0+\tau)^\alpha}{2}|\varphi_\tau|^2 + \frac{(d-2)H_0}{4} \varphi\varphi_\tau +  \frac{(d-2)^2H_0^2}{8(1+\tau)^\alpha}\varphi^2+ \frac{(T_0+\tau)^\alpha}{2}|\nabla\varphi|^2 \right.\right.\nonumber\\
	&\left.\left. + \frac{\varepsilon a^2(T_0+\tau)^\alpha}{m+1}|\varphi|^{m+1}\right)\right]-\nabla\cdot\left(e^{2\psi}\left((T_0+\tau)^\alpha\varphi_\tau\nabla\varphi+\frac{(d-2)H_0}{4} \varphi \nabla\varphi\right)\right)\nonumber\\
	& + e^{2\psi}\left[\left(\frac{(d-2)H_0}{2}-(T_0+\tau)^\alpha\psi_\tau\right)|\varphi_\tau|^2 +\left(\frac{(d-2)H_0}{8}-\psi_\tau(T_0+\tau)^\alpha\right)|\nabla\varphi|^2 \right.\nonumber\\
	&+\frac{(d-2)H_0}{4}\left(\xi(d-1)\left(\frac{-2H_0 \alpha}{(1+\tau)^{1+\alpha}}+\frac{(d-2)H_0^2}{(1+\tau)^{2\alpha}}\right)+\frac{(d-2)H_0}{(1+\tau)^\alpha}\left(-\psi_\tau+\frac{\alpha}{2(1+\tau)}\right)\right)|\varphi|^2 \nonumber\\
	& \left. +\left(\frac{\varepsilon a^2(d-2)H_0}{8} - \frac{2\varepsilon a}{m+1}(a\psi_\tau + a_\tau)(T_0+\tau)^\alpha \right)|\varphi|^{m+1}\right]\nonumber \\
	& +\xi (d-1)\left(-\frac{2H_0\alpha}{(1+\tau)}+\frac{(d-2)H_0^2}{(1+\tau)^\alpha}\right)e^{2\psi}\varphi\varphi_\tau\nonumber\\
	&+\frac{(d-2)H_0}{2}e^{2\psi}\left(-\psi_\tau\varphi\varphi_\tau + \varphi \nabla\varphi \cdot\nabla\psi\right) \leq 0.
	\end{align}
	
	Using Young's inequality,
	\begin{align}
	|\varphi\varphi_\tau|\leq\frac{1}{2}\left(\left|\varphi\right|^2+\left|\varphi_\tau\right|^2\right),\qquad
	\left|\varphi \nabla\varphi\cdot\nabla\psi \right|\leq \frac{1}{2}\left(|\nabla\psi|^2|\varphi|^2 + |\nabla\varphi|^2\right),
	\end{align}
	and taking integration over $\mathbb{R}^{d-1}$, yields
	\begin{align}
	&\frac{\mathrm{d}}{\mathrm{d}\tau} \int_{\mathbb{R}^{d-1}}\left[e^{2\psi}\left(\frac{(T_0+\tau)^\alpha}{2}|\varphi_\tau|^2 + \frac{(d-2)H_0}{4} \varphi\varphi_\tau +  \frac{(d-2)^2H_0^2}{8(1+\tau)^\alpha}|\varphi|^2+ \frac{(T_0+\tau)^\alpha}{2}|\nabla\varphi|^2 \right.\right.\nonumber\\
	&\left.\left. + \frac{\varepsilon a^2(T_0+\tau)^\alpha}{m+1}|\varphi|^{m+1}\right)\right]\mathrm{d}x -\int_{\mathbb{R}^{d-1}}\nabla\cdot\left(e^{2\psi}\left((T_0+\tau)^\alpha\varphi_\tau\nabla\varphi+\frac{(d-2)H_0}{4} \varphi \nabla\varphi\right)\right)\mathrm{d}x\nonumber\\
	& + \int_{\mathbb{R}^{d-1}} e^{2\psi}\left[\left(\frac{(d-2)H_0}{2}+\frac{\xi (d-1)}{2}\left(-\frac{2H_0\alpha}{(1+\tau)}+\frac{(d-2)H_0^2}{(1+\tau)^\alpha}\right)\right.\right.\nonumber\\
	& \left.-\psi_\tau\left( \frac{(d-2)H_0}{4}+(T_0+\tau)^\alpha\right)\right)|\varphi_\tau|^2+\left(\frac{3(d-2)H_0}{8}-\psi_\tau(T_0+\tau)^\alpha\right) |\nabla\varphi|^2\nonumber\\
	&+\left(\frac{(d-2)H_0}{4}\left(\xi(d-1)\left(\frac{-2H_0 \alpha}{(1+\tau)^{1+\alpha}}+\frac{(d-2)H_0^2}{(1+\tau)^{2\alpha}}\right)+\frac{(d-2)H_0}{(1+\tau)^\alpha}\left(-\psi_\tau+\frac{\alpha}{2(1+\tau)}\right)\right)\right. \nonumber\\
	& \left.+\frac{\xi (d-1)}{2}\left(-\frac{2H_0\alpha}{(1+\tau)}+\frac{(d-2)H_0^2}{(1+\tau)^\alpha}\right)-\frac{(d-2)H_0}{4}\psi_\tau+ \frac{(d-2)H_0}{4} |\nabla\psi|^2\right)|\varphi|^2 \nonumber\\
	& \left. +\left(\frac{\varepsilon a^2(d-2)H_0}{8} - \frac{2\varepsilon a}{m+1}(a\psi_\tau + a_\tau)(T_0+\tau)^\alpha \right)|\varphi|^{m+1}\right]\mathrm{d}x\leq 0.
	\end{align}
	
	Under the assumption that $\varphi(\tau,\cdot)$ has compact support (or decays sufficiently fast at spatial infinity), the boundary integral vanishes
	\begin{align}
	\int_{\mathbb{R}^{d-1}}\nabla\cdot\left(e^{2\psi}\left((T_0+\tau)^\alpha\varphi_\tau\nabla\varphi+\frac{(d-2)H_0}{4} \varphi \nabla\varphi\right)\right)\mathrm{d}x=0,
	\end{align}
	so that the flux term does not contribute to the integrated inequality.
	
	In addition, a persistent difficulty in establishing weighted energy estimates for scalar field dynamics in expanding spacetimes stems from the nonlinear potential term, which explicitly involves the cosmological scale factor $a(\tau)$. In contrast to the constant-coefficient setting	studied in \cite{Nishihara}, the scale factor here grows monotonically in an inflationary regime, reflecting the accelerated expansion of the FLRW background. This growth amplifies the nonlinear source term and may counteract the dissipative effect of the Hubble friction, encoded in the time-dependent damping coefficient. This motivates the introduction of a cosmological rescaling
	\begin{align}\label{phi}
	\phi(\tau,x):=a(\tau)^{\frac{d-2}{2}}\varphi(\tau,x),
	\end{align}
	that absorbs the scale factor into the field amplitude and extracts the essential time dependence of the nonlinear
	interaction. Under this transformation, the weighted energy identity derived above can be rewritten in the form
	\begin{align}
	&\frac{\mathrm{d}}{\mathrm{d}\tau}\int_{\mathbb{R}^{d-1}}e^{2\psi}\left[a^{-(d-2)}\left(\frac{(T_0+\tau)^\alpha}{2}|\phi_\tau|^2 +\frac{(T_0+\tau)^\alpha}{2}|\nabla\phi|^2 + \frac{(T_0+\tau)^\alpha}{(1+\tau)^{2\alpha}}\frac{(d-2)H_0}{8}|\phi|^2 \right.\right.\nonumber\\
	&\left.\left.+\frac{(d-2)H_0}{4}\left(1-2\left(\frac{T_0+\tau}{1+\tau}\right)^\alpha\right)\phi\phi_\tau\right)+ \frac{\varepsilon a^2(T_0+\tau)^\alpha}{m+1}a^{-\frac{(d-2)(m+1)}{2}}|\phi|^{m+1}\right]\mathrm{d}x\nonumber\\
	&+\int_{\mathbb{R}^{d-1}} e^{2\psi}\left[a^{-(d-2)}\left(\left(1-\frac{(d-2)H_0}{2(1+\tau)^\alpha}\right)\left(\frac{(d-2)H_0}{2}+\frac{\xi (d-1)}{2}\left(-\frac{2H_0\alpha}{(1+\tau)}+\frac{(d-2)H_0^2}{(1+\tau)^\alpha}\right)\right.\right.\right.\nonumber\\
	& \left.-\psi_\tau\left( \frac{(d-2)H_0}{4}+(T_0+\tau)^\alpha\right)\right)|\phi_\tau|^2+\left(\frac{3(d-2)H_0}{8}-\psi_\tau(T_0+\tau)^\alpha\right) |\nabla\phi|^2\nonumber\\
	&+\left(\frac{(d-2)^2H_0^2}{4(1+\tau)^\alpha}\left(\frac{\alpha}{2(1+\tau)}-1+\frac{(d-2)H_0}{2(1+\tau)^\alpha}-\psi_\tau\left(\frac{1}{2}+\left(\frac{T_0+\tau}{1+\tau}\right)^\alpha + \frac{(d-2)H_0}{4(1+\tau)^\alpha}\right)\right)\right.\nonumber\\
	&+\frac{\xi(d-1)}{2}\left(-\frac{2H_0\alpha}{(1+\tau)}+\frac{(d-2)H_0^2}{(1+\tau)^\alpha}\right)\left(1+\frac{(d-2)^2H_0^2}{4(1+\tau)^{2\alpha}}\right)+\frac{(d-2)H_0}{2}\psi_\tau\left(\frac{T_0+\tau}{1+\tau}\right)^\alpha\nonumber\\
	&\left.\left.-\frac{(d-2)H_0}{4}\psi_\tau+ \frac{(d-2)H_0}{4} |\nabla\psi|^2\right)|\phi|^2\right)\nonumber\\
	& \left. +\left(\frac{\varepsilon a^2(d-2)H_0}{8} - \frac{2\varepsilon a}{m+1}(a\psi_\tau + a_\tau)(T_0+\tau)^\alpha \right)a^{-\frac{(d-2)(m+1)}{2}}|\phi|^{m+1}\right]\mathrm{d}x.
	\end{align}	
	The curvature-coupling term $\xi$ produces a time-dependent function $O((1+\tau)^{-2\alpha})$, which decays faster than the cosmological damping $H(\tau)=O((1+\tau)^{-\alpha})$. Hence, for later calculations, it can be absorbed into the lower-order coefficient in the weighted energy functional.
	
	Let us define
	\begin{align}\label{gamma}
	\gamma:=2-\frac{(d-2)(m+1)}{2}=\sigma-(d-2).
	\end{align}
	The decay mechanism governing the asymptotic regime is naturally interpreted as a cosmological redshift suppression of the nonlinear interaction. In an accelerating FLRW background, the expansion induces a systematic dilution of the field amplitude, which, after the cosmological rescaling \eqref{phi}, manifests itself as a time-dependent prefactor multiplying the nonlinear potential. In the parameter range $d\geq 4$ and $1<m<1+\frac{2}{d-1}$, this prefactor takes the form $a^\gamma$ with $\gamma<0$, implying that the effective strength of the self-interaction decays monotonically over time. As a consequence, the nonlinear term becomes asymptotically negligible relative to the principal linear operator, and the long-time dynamics are governed by a dispersive linear evolution. This yields a precise mathematical formulation of the redshift-induced decay estimates for nonlinear scalar field interactions in FLRW backgrounds.
	
	Next, we introduce a constant $C_0>0$, such that the weighted energy inequality can be written as follows
	\begin{align}\label{Eq.11}
	&\frac{\partial}{\partial\tau}\tilde{E}_\psi(\tau;\phi)+H_\psi(\tau;\phi)\leq 0.
	\end{align}	
	where,
	\begin{align}
	\tilde{E}_\psi(\tau;\phi):=&\int_{\mathbb{R}^{d-1}}e^{2\psi}\left[a^{-(d-2)}\left(\frac{(T_0+\tau)^\alpha}{2}|\phi_\tau|^2 +\frac{(T_0+\tau)^\alpha}{2}|\nabla\phi|^2 + \frac{(T_0+\tau)^\alpha}{(1+\tau)^{2\alpha}}\frac{(d-2)H_0}{8}|\phi|^2 \right.\right.\nonumber\\
	&\left.\left.+\frac{(d-2)H_0}{4}\left(1-2\left(\frac{T_0+\tau}{1+\tau}\right)^\alpha\right)\phi\phi_\tau\right)+ \frac{\varepsilon (T_0+\tau)^\alpha}{m+1}a^\gamma|\phi|^{m+1}\right]\mathrm{d}x,
	\end{align}
	and
	\begin{align}
	H_\psi(\tau,\phi):=&C_0\int_{\mathbb{R}^{d-1}} e^{2\psi}\left\{ \left(1-\psi_\tau(T_0+\tau)^\alpha\right)\left[a^{-(d-2)}\left(|\phi_\tau|^2+|\nabla\phi|^2\right)+\frac{\varepsilon a^\gamma}{m+1}|\phi|^{m+1}\right]\right.\nonumber\\
	&\left.+\frac{\varepsilon a^\gamma}{m+1}|\phi|^{m+1}-a^{-(d-2)}\psi_\tau H_0(1+\tau)^{-\alpha}|\phi|^2\right\}\mathrm{d}x.
	\end{align}
	We also define
	\begin{align}\label{Ebar}
	\bar{E}_\psi(\tau;\phi):=&\int_{\mathbb{R}^{d-1}}e^{2\psi}\left\{(T_0+\tau)^\alpha \left[a^{-(d-2)}\left(|\phi_\tau|^2 + |\nabla\phi|^2\right) + \frac{\varepsilon a^\gamma}{m+1}|\phi|^{m+1}\right] \right.\nonumber\\
	&\left.+a^{-(d-2)}H_0(1+\tau)^{-\alpha}|\phi|^2 \right\}\mathrm{d}x.
	\end{align}
	Then there exist constants $C_1,C_2>0,$ independent of $\tau$, such that
	\begin{align}
	C_1 \bar{E}_\psi(\tau;\phi) \leq \tilde{E}_\psi(\tau;\phi) \leq C_2 \bar{E}_\psi(\tau;\phi).
	\end{align}
	Multiplying \eqref{Eq.11} by $(T_0+\tau)^k$ for a fixed $k>0$, we obtain
	\begin{align} \label{Eq.12}
	\frac{\mathrm{d}}{\mathrm{d}\tau}\left[(T_0+\tau)^k\tilde{E}_\psi (\tau;\phi)\right] + (T_0+\tau)^k \left[H_\psi(\tau;\phi)-\frac{k}{(T_0+\tau)}\tilde{E}_\psi (\tau;\phi)\right]\leq 0.
	\end{align}
	Using the equivalence between $\tilde{E}_\psi(\tau;\phi)$ and $\bar{E}_\psi(\tau;\phi)$, we estimate
	\begin{align}
	H_\psi(\tau;\phi)-\frac{k}{T_0+\tau}\tilde E_\psi(\tau;\phi)
	\geq H_\psi(\tau;\phi)-\frac{kC_2}{T_0+\tau}\bar{E}_\psi(\tau;\phi)\geq I_1+I_2,
	\end{align} 
	where
	\begin{align}
	I_1:=&\int_{\mathbb{R}^{d-1}} e^{2\psi}\left[C_0\left(1-\psi_\tau(T_0+\tau)^\alpha\right)-\frac{kC_2(T_0+\tau)^\alpha}{(T_0+\tau)}\right]\nonumber\\
	&\times\left[a^{-(d-2)}\left(|\phi_\tau|^2 + |\nabla\phi|^2\right)+\frac{\varepsilon a^\gamma}{m+1}|\phi|^{m+1}\right]\mathrm{d}x,\\
	I_2:=&\int_{\mathbb{R}^{d-1}} e^{2\psi}\left[-\frac{kC_2}{(T_0+\tau)}a^{-(d-2)} H_0 (1+\tau)^{-\alpha}|\phi|^2 - C_0 a^{-(d-2)} \psi_\tau H_0(1+\tau)^{-\alpha}|\phi|^2\right.\nonumber\\
	&\left. + C_0\frac{\varepsilon a^\gamma}{m+1}|\phi|^{m+1}\right]\mathrm{d}x.
	\end{align}
	
	We fix $T_0\gg 1$ sufficiently large so that $\frac{1}{2}C_0 > \frac{kC_2}{T_0^{1-\alpha}}$. Since $\psi_\tau\leq 0$ and $(T_0+\tau)^\alpha(-\psi_\tau)$ can be made uniformly small by taking $T_0$ large, we obtain
	\begin{align}\label{Eq.13}
	I_1\geq \frac{C_0}{2}\int_{\mathbb{R}^{d-1}} e^{2\psi}\left(1-\psi_\tau(T_0+\tau)^\alpha\right)\left[a^{-(d-2)}\left(|\phi_\tau|^2 + |\nabla\phi|^2\right)+ \frac{\varepsilon a^\gamma}{m+1}|\phi|^{m+1}\right]\mathrm{d}x.
	\end{align}
	For $d\geq 4$ and $1<m$, the exponent $\gamma$ defined in \eqref{gamma} is negative, so that $a^\gamma\leq 1$ and the weighted nonlinear term is uniformly dominated by the standard energy density. Consequently, the integrand in $I_1$ is nonnegative and controlled by the basic weighted energy, which implies that $I_1$ is uniformly bounded in time.
	
	We next introduce the spatial decomposition
	\begin{align}
	\Omega_\kappa := \left\{x\in\mathbb{R}^{d-1} : \frac{|x|^2}{(T_0+\tau)^{1+\alpha}}\geq \kappa \right\},\qquad \Omega_\kappa^c := \mathbb{R}^{d-1}\setminus \Omega_\kappa,
	\end{align}
	where $\kappa>0$ will be chosen sufficiently large, and write
	\begin{align}
	I_2 =& \int_{\Omega_\kappa^c} e^{2\psi} \left\{\left(\frac{C_0 b(1+\alpha)\kappa}{(T_0+\tau)}-\frac{kC_2}{(T_0+\tau)}\right)a^{-(d-2)}H_0(1+\tau)^{-\alpha}|\phi|^2+\frac{C_0\varepsilon a^\gamma}{m+1}|\phi|^{m+1}\right\}\mathrm{d}x\nonumber\\
	&+ \int_{\Omega_\kappa} e^{2\psi} \left\{\left(\frac{C_0 b(1+\alpha)\kappa}{(T_0+\tau)}-\frac{kC_2}{(T_0+\tau)}\right)a^{-(d-2)}H_0(1+\tau)^{-\alpha}|\phi|^2+\frac{C_0\varepsilon a^\gamma}{m+1}|\phi|^{m+1}\right\}\mathrm{d}x\nonumber\\
	&=I_{21}+I_{22}.
	\end{align}
	By taking $\kappa\gg 1$ sufficiently large, we obtain
	\begin{align}\label{Eq.14}
	I_{21}\geq& \int_{\Omega_\kappa^c} e^{2\psi} \left(\frac{C_0 b(1+\alpha)\kappa}{(T_0+\tau)}-\frac{kC_2}{(T_0+\tau)}\right)a^{-(d-2)}H_0(1+\tau)^{-\alpha}|\phi|^2\mathrm{d}x \nonumber\\
	&+C_0\int_{\Omega_\kappa^c} e^{2\psi}\frac{\varepsilon a^\gamma}{m+1}|\phi|^{m+1}\mathrm{d}x\nonumber\\
	\geq& C_0\int_{\Omega_\kappa^c} e^{2\psi}\frac{\varepsilon a^\gamma}{m+1}|\phi|^{m+1}\mathrm{d}x \geq 0.
	\end{align}
	On the region $\Omega_\kappa$, we write
	\begin{align}
	I_{22}\geq&\int_{\Omega_\kappa} e^{2\psi}\left(\frac{C_0 b(1+\alpha)\kappa}{(T_0+\tau)}-\frac{kC_2}{(T_0+\tau)}\right)a^{-(d-2)} H_0(1+\tau)^{-\alpha}|\phi|^2 \mathrm{d}x\nonumber\\
	&+ C_0 \int_{\Omega_\kappa} e^{2\psi} \frac{\varepsilon a^\gamma}{m+1}|\phi|^{m+1}\mathrm{d}x\nonumber\\
	\geq& \int_{\Omega_\kappa} e^{2\psi} \left(C_0 \frac{\varepsilon a^\gamma}{m+1}|\phi|^{m+1}-\frac{kC_2 H_0}{(1+\tau)^{1+\alpha}}a^{-(d-2)}|\phi|^2\right)\mathrm{d}x.
	\end{align}
	We now apply Young's inequality with 
	\begin{align}
	p=\frac{m+1}{2},\qquad q=\frac{m+1}{m-1}, \qquad \frac{1}{p} + \frac{1}{q} = 1,
	\end{align}
	which yields
	\begin{align}
	\left(a^{\frac{4}{m+1}}a^{-(d-2)}|\phi|^2\right)\left(a^{-\frac{4}{m+1}}kC_2H_0(1+\tau)^{-(1+\alpha)}\right)\leq& \frac{C_0\varepsilon a^\gamma}{2(m+1)}|\phi|^{m+1}\nonumber\\
	&+ C a^{-\frac{4}{(m-1)}}(1+\tau)^{-(1+\alpha)\frac{(m+1)}{(m-1)}},
	\end{align}
	where
	\begin{align}
	C_0=\frac{4}{\varepsilon},\qquad C=\frac{(m-1)}{(m+1)}\left(kC_2 H_0\right)^{\frac{(m+1)}{m-1}}.
	\end{align}
	Therefore, we obtain
	\begin{align}\label{Eq.15}
	I_{22}\geq& \frac{C_0\varepsilon}{2(m+1)}a^\gamma\int_{\Omega_\kappa} e^{2\psi} |\phi|^{m+1}\mathrm{d}x - C a^{-\frac{4}{(m-1)}}(1+\tau)^{-(1+\alpha)\frac{(m+1)}{(m-1)}} \int_{\Omega_\kappa} e^{2\psi}\mathrm{d}x\nonumber\\
	\geq& \frac{C_0\varepsilon}{2(m+1)}a^\gamma\int_{\Omega_\kappa} e^{2\psi} |\phi|^{m+1}\mathrm{d}x -  C a^{-\frac{4}{(m-1)}}(1+\tau)^{-(1+\alpha)\left(\frac{m+1}{m-1}-\frac{(d-1)}{2}\right)}.
	\end{align}
	For $d\geq 4$ and $1<m<1+\frac{2}{d-1}$, we have $\frac{m+1}{m-1}-\frac{(d-1)}{2}>0$,
	so the second term on the right-hand side decays in time and is uniformly controlled by the basic weighted energy. Hence, $I_{22}$ remains uniformly bounded.
	
	Combining \eqref{Eq.13}, \eqref{Eq.14}, and \eqref{Eq.15},  we conclude that
	\begin{align} \label{Eq.16}
	(T_0+\tau)^k \left[H_\psi(\tau;\phi)-\frac{k}{(T_0+\tau)}\tilde{E}_\psi (\tau;\phi)\right]\geq& C_3(T_0+\tau)^k\int_{\mathbb{R}^{d-1}}e^{2\psi}\left(1-\psi_\tau(1+\tau)^\alpha\right)\nonumber\\
	&\times\left[a^{-(d-2)}\left(|\phi_\tau|^2 + |\nabla\phi|^2\right)+ \frac{\varepsilon a^\gamma}{m+1}|\phi|^{m+1}\right]\mathrm{d}x\nonumber\\
	& - C_4 a^{-\frac{4}{(m-1)}}(T_0+\tau)^k (1+\tau)^{-(1+\alpha)\left(\frac{m+1}{m-1}-\frac{(d-1)}{2}\right)},
	\end{align}
	for some constants $C_3,C_4>0$. Therefore, from \eqref{Eq.12}, we can infer
	\begin{align}\label{Eq.17}
	&\frac{\mathrm{d}}{\mathrm{d}\tau}\left[(T_0+\tau)^k\tilde{E}_\psi (\tau;\phi)\right]+C_3(T_0+\tau)^k\int_{\mathbb{R}^{d-1}}e^{2\psi}\left(1-\psi_\tau(1+\tau)^\alpha\right)\nonumber\\ 
	&\times\left[a^{-(d-2)}\left(|\phi_\tau|^2 + |\nabla\phi|^2\right)+ \frac{\varepsilon a^\gamma}{m+1}|\phi|^{m+1}\right]\mathrm{d}x\leq C_4 a^{-\frac{4}{(m-1)}}(T_0+\tau)^k (1+\tau)^{-(1+\alpha)\left(\frac{m+1}{m-1}-\frac{(d-1)}{2}\right)}.
	\end{align}
	Moreover, there exists a constant $C_5>1$ such that
	\begin{align}	\frac{1}{C_5}\frac{1}{T_0+\tau}
	\leq \frac{1}{1+\tau} \leq \frac{C_5}{T_0+\tau}.
	\end{align}
	We fix $0<\epsilon<1$ such that
	\begin{align}
	k-(1+\alpha)\left(\frac{m+1}{m-1}-\frac{(d-1)}{2}\right)=-1+\epsilon.
	\end{align}
	With this choice of $k$, integrating (\ref{Eq.17}) over the time interval $[0,\tau]$, yields
	\begin{align}
	&(1+\tau)^k\tilde{E}_\psi(\tau;\phi) + \int_{0}^{\tau}C_3(1+t)^k\int_{\mathbb{R}^{d-1}}e^{2\psi}\left(1-\psi_t(1+t)^\alpha\right)\nonumber\\ 
	&\times\left[a^{-(d-2)}\left(|\phi_t|^2 + |\nabla\phi|^2\right)+ \frac{\varepsilon a^\gamma}{m+1}|\phi|^{m+1}\right]\mathrm{d}x\mathrm{d}t \leq C_6 a^{-\frac{4}{(m-1)}}(1+\tau)^\epsilon,
	\end{align}
	for some constant $C_6$ independent of $\tau$. Consequently,
	\begin{align}
	\tilde{E}_\psi(\tau;\phi) \leq C_6a^{-\frac{4}{(m-1)}}(1+\tau)^{-(1+\alpha)\left(\frac{m+1}{m-1}-\frac{(d-1)}{2}\right)+1}.
	\end{align}
	Invoking \eqref{Ebar}, we further obtain
	\begin{align}
	\int_{\mathbb{R}^{d-1}}e^{2\psi}a^{-(d-2)}H_0(1+\tau)^{-\alpha}|\phi|^2\mathrm{d}x\leq C_6a^{-\frac{4}{(m-1)}}(1+\tau)^{-(1+\alpha)\left(\frac{m+1}{m-1}-\frac{(d-1)}{2}\right)+1},
	\end{align}
	which implies
	\begin{align}
	\int_{\mathbb{R}^{d-1}}e^{2\psi}|\phi|^2\mathrm{d}x\leq C_6 a^{-2\left(\frac{2}{m-1}-\frac{(d-2)}{2}\right)}(1+\tau)^{-(1+\alpha)\left(\frac{2}{m-1}-\frac{(d-1)}{2}\right)}.
	\end{align}
	Recalling the cosmological rescaling \eqref{phi}, we obtain
	\begin{align}
	\int_{\mathbb{R}^{d-1}}e^{2\psi}|\varphi|^2\mathrm{d}x\leq C_6 a^{-\frac{4}{m-1}}(1+\tau)^{-(1+\alpha)\left(\frac{2}{m-1}-\frac{(d-1)}{2}\right)}.
	\end{align}
	Consequently,
	\begin{align}
	\|\varphi\|_{L^2(\mathbb{R}^{d-1})}\leq Ca^{-\frac{2}{m-1}}(1+\tau)^{-(1+\alpha)\left(\frac{1}{m-1}-\frac{d-1}{4}\right)}.
	\end{align}
	To pass to the $L^1-$ norm, we invoke the weight \eqref{weight function} and the Cauchy-Schwarz inequality
	\begin{align}
	\|\varphi\|_{L^1(\mathbb{R}^{d-1})}=&\int_{\mathbb{R}^{d-1}} e^{\psi}|\varphi|\cdot e^{-\psi}\mathrm{d}x\leq \left(\int_{\mathbb{R}^{d-1}}e^{2\psi}|\varphi|^2\mathrm{d}x\right)^\frac{1}{2}\left(\int_{\mathbb{R}^{d-1}} e^{-2\psi}\mathrm{d}x\right)^\frac{1}{2},
	\end{align}
	which yields
	\begin{align}
	\|\varphi\|_{L^1(\mathbb{R}^{d-1})}\leq Ca^{-\frac{2}{m-1}}(1+\tau)^{-(1+\alpha)\left(\frac{1}{m-1}-\frac{d-1}{2}\right)}.
	\end{align}
	Hence, for $d\geq 4$ and nonlinear exponent $1<m<1+\frac{2}{d-1}$, both 
	$\|\varphi(\tau,\cdot)\|_{L^2(\mathbb{R}^{d-1})}$ and 
	$\|\varphi(\tau,\cdot)\|_{L^1(\mathbb{R}^{d-1})}$ 
	decay in time at rates governed by the cosmological redshift factor $a(\tau)^{-\frac{2}{m-1}}$ together with the polynomial decay induced by the time-dependent damping. In particular, accelerated expansion enforces a strong suppression of the nonlinear self-interaction in the sense of weighted $L^1$-$L^2$ decay estimates for spatially localized data. These estimates provide a rigorous formulation of redshift-induced decay estimates for nonlinear scalar field interactions, showing that inflationary expansion dynamically weakens nonlinear self-interactions at late times. This completes the proof.
\end{proof}

\subsection{Discussion of the Intermediate Regime}
Between the superconformal scattering regime $m>m_{\mathrm{conf}}$ and the diffusion-dominated range $1<m<1+\frac{2}{d-1}$,
there exists the intermediate interval $1+\frac{2}{d-1}<m<m_{\mathrm{conf}}.$ This regime is not covered by the results established in the present work and deserves a brief discussion.

First, the scattering mechanism of Theorem \ref{Thm:scattering} no longer applies. Indeed, for $m<m_{\mathrm{conf}}$, the exponent $\sigma=\frac{d+2-(d-2)m}{2}$ is positive, so that the effective coupling $g(\tau)=a(\tau)^\sigma$ grows with the cosmological expansion. Consequently, $g\notin L^1([0,\infty))$, and the Duhamel integral used in the proof of Theorem \ref{Thm:scattering} is no longer absolutely convergent. In this sense, the nonlinearity cannot be regarded as a short-range perturbation of the free wave equation.

Second, the weighted-energy argument developed in Theorem \ref{Thm:diffusion} also ceases to be effective. The key decay estimate relies on the positivity of the exponent $\frac{m+1}{m-1}-\frac{d-1}{2},$ which is equivalent to $m<1+\frac{2}{d-1}.$ For exponents in the present interval, this quantity becomes non-positive, and the resulting weighted-energy inequality no longer yields a decaying upper bound.

Therefore, the intermediate regime represents a transition region in which neither the redshift-induced scattering mechanism nor the diffusion-type decay mechanism dominates the asymptotic dynamics. The competition between cosmological damping and nonlinear amplification is more delicate in this range, and its detailed long-time behavior remains an interesting open problem. A refined analysis may require different techniques, such as Strichartz estimates adapted to expanding backgrounds or alternative weighted-energy constructions.

\section{Implications for Inflationary Stability and Nonlinear Dynamics}\label{Sec4}
The results obtained in this paper provide a physical interpretation of the role of inflation as a dynamic mechanism that not only dampens scalar field fluctuations but also effectively weakens its nonlinear interaction law. The effective coupling $g(\tau)=a(\tau)^{\sigma}$ in the rescaled equation \eqref{Eq:rescaled} shows that redshift suppression is governed by the conformal threshold $m_{\mathrm{conf}}$: for superconformal interactions $m>m_{\mathrm{conf}}$ one has $\sigma<0$ and the coupling decays (indeed $g\in L^1([0,\infty))$ when $0\le\alpha<1$), leading to asymptotic linearization and scattering. At the conformal power $m=m_{\mathrm{conf}}$ the coupling is constant, while for $1<m<m_{\mathrm{conf}}$ the coupling grows, so redshift alone does not suppress nonlinear effects in that regime. From a cosmological point of view, this means that the inflationary phase acts as a natural regulator for field dynamics, stabilizing the evolution against nonlinear instabilities and preventing the formation or persistence of localized nonlinear structures such as lumps or bound states during the rapid expansion phase. A key aspect of the cosmic no-hair principle is that inflationary expansion suppresses anisotropies and inhomogeneities, driving spacetime toward a homogeneous de Sitter-like configuration, as first rigorously shown for homogeneous cosmologies with a positive cosmological constant \cite{Kaloper,Maleknejad} and extended to power-law inflationary settings \cite{Kitada}. Thus, the cosmic no-hair principle acquires a stronger meaning at the field theory level, not only is the geometry of spacetime forced towards a homogeneous and isotropic form, but the dynamics of the scalar field are also driven towards a regime increasingly free from the influence of nonlinear interactions. This finding is relevant for understanding inflationary stability, the dynamics of the inflaton and additional scalar fields, and the validity of effective field descriptions in the early universe, where cosmological expansion actively suppresses nonlinear effects and ensures the long-term controllability of field evolution.

\section*{Acknowledgments}
This research is funded by the Indonesian Endowment Fund for Education (LPDP) on behalf of the Indonesian Ministry of Higher Education, Science and Technology and is managed under the EQUITY Program (Contract No.~4298/B3/DT.03.08/2025). The work of MPW is also partly supported by BLU Research Grant from Universitas Jenderal Soedirman under the RDU (Riset Dasar Unsoed) scheme for the year 2026. Grant No. 13.391/UN23.34/PT.01.00/IV/2026.

\section*{Statements and Declarations}

\paragraph{Funding.} This research is funded by the Indonesian Endowment Fund for Education (LPDP) on behalf of the Indonesian Ministry of Higher Education, Science and Technology and is managed under the EQUITY Program (Contract No.~4298/B3/DT.03.08/2025). The work of MPW is also partly supported by BLU Research Grant from Universitas Jenderal Soedirman under the RDU (Riset Dasar Unsoed) scheme for the year 2026. Grant No. 13.391/UN23.34/PT.01.00/IV/2026.

\paragraph{Competing interests.} The authors declare that they have no competing interests.

\paragraph{Data availability.} No data were generated or analyzed in this study.

\paragraph{Code availability.} Not applicable.


\begin{thebibliography}{00}
	\bibitem{Guth}
	Guth, A.H.: Inflationary universe: A possible solution to the horizon and flatness problems. \emph{Phys. Rev. D.}, \textbf{23} (2), 347-356 (1981).
	
	\bibitem{Mukhanov}
	Mukhanov, V.F., Feldman, H.A., Brandenberger, R.H.: Theory of cosmological perturbations. \emph{Phys. Rep.}, \textbf{215} (5-6), 203-333 (1992).
	
	\bibitem{Frusciante}
	Frusciante, N., Perenon, L.:
	Effective field theory of dark energy: A review. \emph{Phys. Rep.}, \textbf{857} (3), 1-63 (2020).
	
	\bibitem{Weinberg}
	Weinberg, S.: Quantum contributions to cosmological correlations. \emph{Phys. Rev. D.}, \textbf{72} (4), 043514 (2005).
	
	\bibitem{Bartolo}
	Bartolo, N., Komatsu, E., Matarrese, S., Riotto, A.:
	Non-Gaussianity from inflation: theory and observations. \emph{Phys. Rep.}, \textbf{402} (3-4), 103-266 (2004).
	
	\bibitem{Chen}
	Chen, X.: Primordial Non-Gaussianities from Inflation Models. \emph{Adv. Astron.}, 638979 (2010).
	
	\bibitem{Wald}
	Wald, R.M.: Asymptotic behavior of homogeneous cosmological models in the presence of a positive cosmological constant. \emph{Phys. Rev. D.}, \textbf{28} (8), 2118-2120 (1983).
	
	\bibitem{Ringstrom}
	Ringstr\"om, H.: Future stability of the Einstein-non-linear scalar field system. \emph{Invent. Math.}, \textbf{173}, 123-208 (2008).
	
	\bibitem{Todorova}
	Todorova, G., Yordanov, B.:
	Critical Exponent for a Nonlinear Wave Equation with Damping. \emph{J. Differ. Equ.}, \textbf{174} (2), 464-489 (2001).
	
	\bibitem{Nishihara}
	Nishihara, K., Zhai, J.:
	Asymptotic behaviors of solutions for time dependent damped wave equations. \emph{J. Math. Anal. Appl.}, \textbf{360} (2), 412-421 (2009).
	
	\bibitem{Kitada}
	Kitada, Y., Maeda, K.:
	Cosmic no-hair theorem in power-law inflation. \emph{Phys. Rev. D.}, \textbf{45} 1416 (1992).
	
	\bibitem{Bruni}
	Bruni, M., Mena, F.C., Tavakol, R.:
	Cosmic no-hair: nonlinear asymptotic stability of de Sitter universe. \emph{Class. Quantum Grav.}, \textbf{19} L23 (2002).
	
	\bibitem{Kaloper}
	Kaloper, N., Scargill, J.:
	Quantum cosmic no-hair theorem and inflation. \emph{Phys. Rev. D.}, \textbf{99} 103514 (2019).

    \bibitem{Hintz}
	Hintz, N., Vasy, A.:
	Semilinear wave equations on asymptotically de Sitter, Kerr-de Sitter and Minkowski spacetimes. \emph{Anal. PDE}, \textbf{8} (8) (2015).

    \bibitem{Baskin}
	Baskin, D.: Strichartz Estimates on Asymptotically de Sitter Spaces. \emph{Ann. Henri Poincaré}, \textbf{14}, 221-252 (2013).

    \bibitem{Nakamura}
	Nakamura, M.: The Cauchy problem for semi-linear Klein–Gordon equations in de Sitter spacetime. \emph{J. Math. Anal. Appl.}, \textbf{410} (1), 445-454 (2014).

    \bibitem{Yagdjian}
	Yagdjian, K.: The semilinear Klein-Gordon equation in de Sitter spacetime. \emph{Discrete Contin. Dyn. Syst.-S.}, \textbf{2} (3), 679-696 (2009).

    \bibitem{Palmieri}
	Palmieri, A., Takamura, H.: A note on blow-up results for semilinear wave equations in de Sitter and anti-de Sitter spacetimes. \emph{J. Math. Anal. Appl.}, \textbf{514} (1), 126266 (2022).
	
	\bibitem{fiki}
	Akbar, F.T., Gunara, B.E., Iqbal, M., and Susanto, H.:
	Local and global existence of solutions to scalar equations on
	spatially flat universe as a background with non-minimal coupling. \emph{Adv. Theor. Math. Phys.}, \textbf{23} (2018).
	
	\bibitem{Radu}
	Radu, P., Todorova, G., Yordanov, B.:
	Higher order energy decay rates for damped wave equations with variable coefficients. \emph{Discrete Contin. Dyn. Syst.}, \textbf{2} (3), 609-629 (2009).
	
	\bibitem{LuLi}
	Lu, L., Li, S.:
	Higher order energy decay for damped wave equations with variable coefficients. \emph{J. Math. Anal. Appl.}, \textbf{418} (1), 64-78 (2014).
	
	\bibitem{Sobajima}
	Sobajima, M., Wakasugi, Y.:
	Remarks on an elliptic problem arising in weighted energy estimates for wave equations with space-dependent damping term in an exterior domain. \emph{AIMS Math.}, \textbf{2} (1), 1-15 (2016).
	
	\bibitem{Maleknejad}
	Maleknejad, A., Sheikh-Jabbari, M.M.:
	Revisiting cosmic no-hair theorem for inflationary settings. \emph{Phys. Rev. D.}, \textbf{85} (12), 123508 (2012).
	
\end{thebibliography}
\end{document}